\documentclass[12pt]{amsart}
\usepackage{amsthm}
\usepackage{amssymb}
\usepackage{amsmath}
\usepackage{color}
\theoremstyle{plain}
\newtheorem{proposition}{Proposition}
\newtheorem{theorem}[proposition]{Theorem}
\newtheorem{lemma}[proposition]{Lemma}

\theoremstyle{definition}

\theoremstyle{definition}
\newtheorem{example}[proposition]{Example}

\newtheorem{remark}[proposition]{Remark}

\numberwithin{equation}{section}

\numberwithin{proposition}{section}

\gdef\myletter{}

\let\savetheequation\theequation
\def\theequation{\savetheequation\myletter}

\usepackage{color}

\newcommand{\CC}{{\mathbb C}}

\newcommand{\RR}{{\mathbb R}}
\newcommand{\ZZ}{{\mathbb Z}}

\renewcommand{\Im}{\mbox{Im}}

\renewcommand{\Re}{\mbox{Re}}

\renewcommand{\date}{\today}
\def \bar{\overline}
\def \hat{\widehat}

\begin{document}

\vskip 3mm

\title[Polynomial Degree Defined by a Convex Body]{\bf Bernstein-Walsh theory associated to convex bodies and applications to multivariate approximation theory}

\author{L. Bos and N. Levenberg*}{\thanks{*Supported by Simons Foundation grant No. 354549}}

\maketitle
\begin{abstract} We prove a version of the Bernstein-Walsh theorem on uniform polynomial approximation of holomorphic functions on compact sets in several complex variables. Here we consider subclasses of the full polynomial space associated to a convex body $P$. As a consequence, we validate and clarify some observations of Trefethen in multivariate approximation theory.

\end{abstract}

\section{Introduction.}\label{sec:intro} A standard theorem in several complex variables, quantifying the classical Oka-Weil theorem on polynomial approximation -- which itself is the multivariate version of the classical Runge theorem for polynomial approximation in the complex plane -- is the {\it Bernstein-Walsh theorem}:

\begin{theorem} \label{basic} Let $K\subset \CC^d$ be compact, nonpluripolar and polynomially convex with $V_{K}$ continuous. 
Let $R > 1$, and let $\Omega_R := \{ z : V_{K} (z) < \log R \}$.
Let $f$ be continuous on $K$. Then
$$
  \limsup_{n\to \infty} D_n(f,K)^{1/n}\leq 1/R
$$ 
if and only if $f$ is
the restriction to $K$ of a function holomorphic in $\Omega_R$.  
\end{theorem}

Here for $K\subset \CC^d$ compact,
\begin{equation}\label{vkfcn}V_K(z) = \max[0,\sup \{\frac{1}{deg (p)}\log |p(z)|: ||p||_K:=\max_{\zeta \in K}|p(\zeta)| \leq 1\}]\end{equation}
where $p$ is a nonconstant holomorphic polynomial; and for a continuous 
complex-valued function $f$ on $K$, $$
D_n(f,K) 
  := 
  \inf \{ ||f-p_n||_{K}:p_n \in \mathcal P_n \}$$
 where $ \mathcal P_n$ is the space of holomorphic polynomials of degree at most $n$. See \cite{BagLe} for a survey and history of Theorem \ref{basic} in both one and several complex variables.
 
 In Trefethen \cite{T}, the author gives some evidence for why one might consider non-traditional notions of ``degree'' of a polynomial in the setting of multivariate approximation theory. More precisely, for certain functions $f$ on $K=[-1,1]^d$ he  compares the approximation numbers $D_n(f, [-1,1]^d)$ where ``degree'' has three possible meanings: total degree, Euclidean degree, or maximum degree. We clarify this distinction in a more general setting by describing generalizations of the extremal functions $V_K$ associated to subclasses of the full polynomial space $\bigcup_n \mathcal P_n$. Given a convex body $P\subset  (\RR^+)^d=[0,\infty)^d$, following Bayraktar \cite{Bay}, we define a {\it $P-$extremal function} $V_{P,K}$ associated to $K$. In Section \ref{sec:back} we list and prove some basic properties of these functions. We state and prove a generalization of Theorem \ref{basic} in this setting in Section \ref{sec:bwthm}. Section \ref{sec:app} recovers the Trefethen cases for $K=[-1,1]^d$ by taking appropriate $P$ and provides explicit examples of functions $f$ comparing rates of approximation. 

\section{Background: $P-$extremal functions.}\label{sec:back}
 In what follows, {\it we fix a convex body $P\subset (\RR^+)^d$;} i.e., a compact, convex set in $(\RR^+)^d$ with non-empty interior $P^o$. Standard examples include the case where 
\begin{enumerate}
\item $P$ is a non-degenerate convex polytope, i.e., the convex hull of a finite subset of $(\ZZ^+)^d$ in $(\RR^+)^d$ with $P^o\not =\emptyset$; 
\item $P_p:=\{(x_1,...,x_d)\in (\RR^+)^d: (x_1^p+\cdots x_d^p)^{1/p}\leq 1\}$ is the (nonnegative) portion of an $l^p$ ball in $(\RR^+)^d$, $1\leq p\leq \infty$.
\end{enumerate}
As a particular case of (2), with $p=1$ we have $P_1=\Sigma$ where
$$\Sigma:=\{(x_1,...,x_d)\in \RR^d: x_1,...,x_d \geq 0, \ x_1+\cdots x_d\leq 1\}.$$

We will consider convex bodies $P\subset (\RR^+)^d$ with the property that 
\begin{equation}\label{phyp} \Sigma \subset kP \ \hbox{for some} \ k\in \ZZ^+.\end{equation}
Associated with $P$, following \cite{Bay}, we consider the finite-dimensional polynomial spaces 
$$Poly(nP):=\{p(z)=\sum_{J\in nP\cap (\ZZ^+)^d}c_J z^J: c_J \in \CC\}$$
for $n=1,2,...$. Here $J=(j_1,...,j_d)$. In the case $P=\Sigma$ we have $Poly(n\Sigma)=\mathcal P_n$, the usual space of holomorphic polynomials of degree at most $n$ in $\CC^d$. Clearly there exists a minimal positive integer $A=A(P)\geq 1$ such that $P\subset A\Sigma$. Thus
\begin{equation}\label{AP} Poly(nP) \subset A\mathcal P_n=\mathcal P_{An} \ \hbox{for all} \ n.\end{equation}
We let $d_n=$dim$(Poly(nP))$. From (\ref{AP}),
\begin{equation}\label{dn} d_n  \leq \hbox{dim}P_{An} = 0(n^d).\end{equation}

Note it follows from convexity of $P$ that 
$$ p_n \in Poly(nP), \ p_m \in Poly(mP) \Rightarrow p_n \cdot p_m\in Poly((n+m)P).$$
It suffices to verify this for monomials $z^A\in Poly(nP), \ z^B \in Poly(mP)$. Then $A\in nP, \ B\in mP$ 
so $A=na, \ a\in P$ and $B=mb, \ b\in P$. Thus 
$$A+B = na+mb= (n+m)[\frac{n}{n+m}a +\frac{m}{n+m}b]\in (n+m)P.$$

Recall the {\it indicator function} of a convex body $P$ is
$$\phi_P(x_1,...,x_d):=\sup_{(y_1,...,y_d)\in P}(x_1y_1+\cdots x_dy_d).$$
For the $P$ we consider, $\phi_P\geq 0$ on $(\RR^+)^d$ with $\phi_P(0)=0$.  
Define the logarithmic indicator function 
$$H_P(z):=\sup_{J\in P} \log |z^J|:=\phi_P(\log |z_1|,...,\log |z_d|).$$
Here $|z^J|:=|z_1|^{j_1}\cdots |z_d|^{j_d}$ for $J=(j_1,...,j_d)\in P$ (the components $j_k$ need not be integers). From (\ref{phyp}), we have
$$H_P(z)\geq \frac{1}{k}\max_{j=1,...,d}\log^+ |z_j|=\frac{1}{k}\max_{j=1,...,d}[\max(0,\log |z_j|)].$$
We will use $H_P$ to define generalizations of the Lelong classes $L(\CC^d)$, the set of all plurisubharmonic (psh) functions $u$ on $\CC^d$ with the property that $u(z) - \log |z| = 0(1), \ |z| \to \infty$, and 
$$L^+(\CC^d)=\{u\in L(\CC^d): u(z)\geq \log^+|z| + C_u\}$$
where $C_u$ is a constant depending on $u$. We remark that, a priori, for a set $E\subset \CC^d$, one defines the {\it global extremal function}
$$V_E(z):=\sup \{u(z): u\in L(\CC^d), \ u\leq 0 \ \hbox{on} \ E\}.$$
It is a theorem, due to Siciak and to Zaharjuta (cf., Theorem 5.1.7 in \cite{K}), that for $K\subset \CC^d$ compact, $V_K$ coincides with the function in (\ref{vkfcn}). Moreover, 
$$V_K^*(z):=\limsup_{\zeta\to z}V_K(\zeta)\in L^+(\CC^d)$$
precisely when $K$ is {\it nonpluripolar}; i.e., for $K$ such that $u$ plurisubharmonic on a neighborhood of $K$ with $u=-\infty$ on $K$ implies $u\equiv -\infty$. 

Define
$$L_P=L_P(\CC^d):= \{u\in PSH(\CC^d): u(z)- H_P(z) =0(1), \ |z| \to \infty \},$$ and 
$$L_{P,+}=L_{P,+}(\CC^d)=\{u\in L_P(\CC^d): u(z)\geq H_P(z) + C_u\}.$$
Then $L_{\Sigma} = L(\CC^d)$ and $L_{\Sigma,+} = L^+(\CC^d)$. Given $E\subset \CC^d$, the {\it $P-$extremal function of $E$} is given by $V^*_{P,E}(z):=\limsup_{\zeta \to z}V_{P,E}(\zeta)$ where
$$V_{P,E}(z):=\sup \{u(z):u\in L_P(\CC^d), \ u\leq 0 \ \hbox{on} \ E\}.$$
For $P=\Sigma$, we recover $V_E=V_{\Sigma,E}$. We will restrict to the case where $E=K\subset \CC^d$ is compact. In this case, Bayraktar \cite{Bay} proved a Siciak-Zaharjuta type theorem showing that $V_{P,K}$ can be obtained using polynomials. Note that $\frac{1}{n}\log |p_n|\in L_P$ for $p_n\in Poly(nP)$. 

\begin{proposition} \label{turgay2} Let $K\subset \CC^d$ be compact and nonpluripolar. Then 
$$V_{P,K} =\lim_{n\to \infty} \frac{1}{n} \log \Phi_n$$
pointwise on $\CC^d$ where
$$\Phi_n(z):= \sup \{|p_n(z)|: p_n\in Poly(nP),  \ ||p_n||_K\leq 1\}.$$
If $V_{P,K}$ is continuous, the convergence is locally uniform on $\CC^d$.

\end{proposition}

\noindent It follows that $V_{P,K}=V_{P,\hat K}$ where 
$$\hat K=\{z:|p(z)|\leq ||p||_K, \ \hbox{for all} \ p\in \bigcup_n \mathcal P_n\}$$
is the polynomial hull of $K$. Also, either $V_{P,K}^*\equiv +\infty$, which occurs if and only if $K$ is pluripolar, or $V_{P,K}^*\in L_{P,+}$. 

\begin{example} \label{torus} Let $K=T^d=\{(z_1,...,z_d):|z_1|=\cdots =|z_d|=1\}$, the unit $d-$torus in $\CC^d$. Then 
$$V_{P,T^d}(z)=H_P(z)= \max_{J\in P} \log |z^J|.$$
This is Example 2.3 in \cite{Bay}. There it is stated only for $P$ a convex polytope. We give an alternate proof in Proposition \ref{productproperty}.

\end{example}

From Proposition \ref{turgay2} we have a Bernstein-Walsh inequality.

\begin{proposition} \label{turgay4} Let $K$ be nonpluripolar. Then for $p_n\in Poly(nP)$, 
\begin{equation}\label{bwestpn} |p_n(z)|\leq ||p_n||_K\exp(nV_{P,K}(z)), \ z\in \CC^d. \end{equation}
\end{proposition}

\noindent In particular, if $V_{P,K}$ is continuous, for $R > 1$
$$ 
  \Omega_R := \{ z : V_{P,K} (z) < \log R \}
$$ 
is an open neighborhood of $K$ and for $p_n\in Poly(nP)$, 
$$
  |p_n(z)| \leq ||p_n||_K R^n, 
  \qquad   z \in \Omega_R.  
$$

Many of the results in Chapter 5 of \cite{K} remain valid for $P-$extremal functions. From the definition of $V_{K,P}$ and Example \ref{torus}, we obtain:

\begin{enumerate} 
\item If $\{K_j\}$ are compact sets with $K_{j+1}\subset K_j$ and $K:=\bigcap_j K_j$, then $\lim_{j\to \infty} V_{P,K_j}=V_{P,K}$;
\item for $K$ compact, $V_{P,K}$ is lower semicontinuous;
\item for $K$ compact, if $V_{P,K}^*|_K=0$ then $V_{P,K}$ is continuous (on $\CC^d$);
\item for $K$ compact, $\lim_{\epsilon \to 0} V_{P,K_{\epsilon}}=V_{P,K}$ where $K_{\epsilon}:=\{z:{\rm{dist}}(z,K)\leq \epsilon\}$.
\end{enumerate}

For $P=\Sigma$, we say a compact set $K$ is {\it $L-$regular} if $V_K$ is continuous on $K$; i.e., if $V_K=V_K^*$ (equivalently, $V_K^*=0$ on $K$). Given a convex body $P\subset (\RR^+)^d$, we call a compact set {\it $PL-$regular} if $V_{P,K}$ is continuous on $K$; i.e., if $V_{P,K}=V_{P,K}^*$. Since $ P \subset A\Sigma$, for each $n$ we have
$$\sup\{\frac{1}{n}\log |p_n|:p_n\in Poly(nP), \ ||p_n||_K\leq 1\}$$
$$\leq \sup\{\frac{1}{n}\log |p_n|:p_n\in \mathcal P_{An}, \ ||p_n||_K\leq 1\}.$$ 
Hence $V_{P,K}(z)\leq A \cdot V_K(z), \ z\in \CC^d$. It follows that if $K$ is $L-$regular, then $K$ is $PL-$regular for any convex body $P\subset (\RR^+)^d$.  

The proofs of Propositions 5.3.12, 5.3.14 and Corollary 5.3.13 in \cite{K} carry over in this setting to show that $K=\bar D$ is $PL-$regular for $D$ a bounded open set with $C^1$ boundary. Hence for any compact set $K$, one can find a decreasing sequence of bounded open sets $\{D_j\}$ with smooth boundaries (thus $\bar D_j$ is $PL-$regular) such that $K=\bigcap \bar D_j$. 

We end this section with a result on $P-$extremal functions for product sets. This will be useful in Section \ref{sec:app}. Before proceeding we require a definition: we call a convex body $P\subset (\RR^+)^d$ a {\it lower set} if for each $n=1,2,...$, whenever $(j_1,...,j_d) \in nP\cap (\ZZ^+)^d$ we have $(k_1,...,k_d) \in nP\cap (\ZZ^+)^d$ for all $k_l\leq j_l, \ l=1,...,d$.  

\begin{proposition} \label{productproperty} Let $P\subset (\RR^+)^d$ be a lower set and let $E_1,...,E_d\subset \CC$ be compact and nonpolar. Then
\begin{equation}\label{prodprop} V^*_{P,E_1\times \cdots \times E_d}(z_1,...,z_d)=\phi_P(V^*_{E_1}(z_1),...,V^*_{E_d}(z_d)).\end{equation}

\end{proposition}

\begin{proof} For simplicity, we do the case $d=2$. Thus let $E,F\subset \CC$ be compact sets and let $\phi_P=\phi_P(x_1,x_2)$ be the support function of $P$. From properties (1) and (4) of $P-$extremal functions, we can assume that $E,F$ are $L-$regular in $\CC$ with $\hat E=E, \ \hat F= F$ and that $E\times F$ is $PL-$regular in $\CC^2$.

To see that 
$$\phi_P(V_{E}(z),V_{F}(w))\leq V_{P,E\times F}(z,w),$$
since $\phi_P(0,0)=0$, it suffices to show that $\phi_P(V_{E}(z),V_{F}(w))\in L_P(\CC^2)$. From the definition of $\phi_P$, 
$$\phi_P(V_{E}(z),V_{F}(w))=\sup_{(x,y)\in P} [xV_E(z)+yV_F(w)]$$
which is an upper envelope of locally bounded above plurisubharmonic functions. As $\phi_P$ is convex and $V_E, V_F$ are continuous, $\phi_P(V_{E}(z),V_{F}(w))$ is continuous. Finally, since $V_E(z)=\log |z|+0(1)$ as $|z|\to \infty$ and $V_F(w)=\log |w|+0(1)$ as $|w|\to \infty$, it follows that 
$\phi_P(V_{E}(z),V_{F}(w))\in L_P(\CC^2)$. 

To prove the reverse inequality, and hence (\ref{prodprop}), we modify the proof of Theorem 5.1.8 in \cite{K}. 
Let $\mu,\nu$ be probability measures on $E,F$ such that $(E,\mu)$ and $(F,\nu)$ are {\it Bernstein-Markov} pairs: thus, given $\epsilon >0$, there exists a positive constant $M$ such that 
\begin{equation}\label{uniBW}||t_n||_E\leq M(1+\epsilon)^n ||t_n||_{L^2(\mu)} \ \hbox{and} \ ||t_n||_F\leq M(1+\epsilon)^n ||t_n||_{L^2(\nu)}\end{equation}
for all $t_n\in \mathcal P_n(\CC)$. Any compact set $B\subset \CC$ admits a measure $\eta$ so that $(B,\eta)$ is a Bernstein-Markov pair; cf., \cite{PELD}. Let $f=f(z,w)\in Poly(nP)$ with $||f||_{E\times F}\leq 1$. Given an orthonormal basis $\{p_j=p_j(z)\}$ in $L^2(\mu)$ and $\{q_k=q_k(w)\}$ in $L^2(\nu)$ for the univariate polynomials, where deg$(p_j)=j$ and deg$(q_k)=k$, we can write 
$$f(z,w)=\sum_{(j,k)\in nP} c_{jk}p_j(z)q_k(w)$$
where
$$|c_{jk}|=|<f, p_jq_k>_{L^2(\mu\times \nu)}|\leq 1$$
for all $(j,k)\in nP$. The lower set property of $P$ implies that $p_jq_k\in Poly(nP)$ for $(j,k)\in nP$. Using (\ref{uniBW}), 
$$||p_j||_E\leq M(1+\epsilon)^j, \ ||q_k||_F\leq M(1+\epsilon)^k$$
so that, using the univariate Bernstein-Walsh estimates, i.e., (\ref{bwestpn}) for $E,F$ ($d=1$), 
$$|p_j(z)q_k(w)|\leq M^2(1+\epsilon)^{j+k}e^{jV_E(z)+kV_F(w)}$$
for all $(z,w)\in \CC^2$. Thus, using (\ref{AP}), 
$$|f(z,w)|\leq  d_nM^2(1+\epsilon)^{An}\cdot \max_{(j,k)\in nP}e^{jV_E(z)+kV_F(w)}.$$
Now 
$$\max_{(j,k)\in nP}e^{jV_E(z)+kV_F(w)}=\bigl(\max_{(j/n,k/n)\in P}e^{\frac{j}{n}V_E(z)+\frac{k}{n}V_F(w)}\bigr)^n$$
so that, since 
$$\log \bigl(\max_{(j/n,k/n)\in P}e^{\frac{j}{n}V_E(z)+\frac{k}{n}V_F(w)}\bigr)= \max_{(j/n,k/n)\in P}[{\frac{j}{n}V_E(z)+\frac{k}{n}V_F(w)}],$$
we have
$$\frac{1}{n}\log |f(z,w)|\leq \frac{1}{n}\log (d_nM^2) + A\log (1+\epsilon) + \max_{(j/n,k/n)\in P}[{\frac{j}{n}V_E(z)+\frac{k}{n}V_F(w)}]$$
$$\leq \frac{1}{n}\log (d_nM^2) + A\log (1+\epsilon) + \phi_P(V_E(z),V_F(w)).$$
The result follows, using (\ref{dn}), upon letting $n\to \infty$.

\end{proof}

\noindent In particular, this gives (another) proof of Example \ref{torus}. 

\begin{remark} \label{rmk25} For $x,y\in \RR^d$, let $x\cdot y=x_1y_1+\cdots x_dy_d$. Let 
$$P^o:=\{y\in \RR^d: \sup_{x\in P}|x\cdot y|\leq 1\}$$
be the {\it polar} of $P$ and let  
$||\cdot||_{P^o}$ be the dual norm defined by $P^o$; i.e., 
$$||(y_1,...,y_d)||_{P^o}:= \sup_{x\in P}|x\cdot y|.$$
Thus $P^o$ is the unit ball in this norm. Then we can write (\ref{prodprop}) as
\begin{equation}\label{product}V_{P,E_1\times \cdots \times E_d}(z_1,...,z_d)=||(V_{E_1}(z_1),...,V_{E_d}(z_d))||_{P^o}.\end{equation}

\end{remark}

\section{Bernstein-Walsh theorem.} \label{sec:bwthm} As in the previous section, we fix a convex body $P\subset (\RR^+)^d$. We prove a Bernstein-Walsh theorem in this setting. Given a compact set $K$, for a continuous 
complex-valued function $f$ on $K$ we define $$
  D_n = D_n(f,K,P) 
  \equiv 
  \inf \{ ||f-p_n||_{K}:p_n \in Poly(nP)\}.
$$

\begin{theorem} \label{BWthm} Let $K$ be compact and $PL-$regular. Let $R > 1$, and let $\Omega_R := \{ z : V_{P,K} (z) < \log R \}$. Let $f$ be continuous on $K$. 
\begin{enumerate}
\item If $P$ is a lower set and $f$ is
the restriction to $K$ of a function holomorphic in $\Omega_R$, then
$$
  \limsup_{n\to \infty} D_n(f,P,K)^{1/n}\leq 1/R.
$$ 
\item If $K=\hat K$, $
  \limsup_{n\to \infty} D_n(f,P,K)^{1/n}\leq 1/R$ implies $f$ is
the restriction to $K$ of a function holomorphic in $\Omega_R$.
\end{enumerate}
\end{theorem}

\begin{proof} (2). Suppose
$$\limsup_{n\to \infty} D_n^{1/n}= 1/R$$
for some $R>1$. We show that if $p_n\in Poly(nP)$ satisfies 
$D_n=||f-p_n||_K$, then the series 
$p_0 + \sum_1^{\infty} (p_n-p_{n-1})$ converges
uniformly on compact subsets of $\Omega_R$ to a holomorphic
function $F$ which agrees with $f$ on
$K$.  To this end, choose $R'$ with 
$1< R' < R$; by hypothesis the polynomials
$p_n$ satisfy
\begin{equation}\label{f=F}
  ||f-p_n||_K \leq {M \over  {R'} ^n}, \qquad n=0,1,2,...,
\end{equation}
for some $M > 0$.  Now let $1 < \rho <  R'$, and apply 
(\ref{bwestpn}) to the polynomial $p_n - p_{n-1}\in Poly(nP)$ 
to obtain 
$$
  \sup _{\bar \Omega_{\rho}} |p_n(z) - p_{n-1} (z)|
  \leq 
  \rho^n ||p_n-p_{n-1}||_K$$
$$  \leq
  \rho^n ( ||p_n -  f||_K + ||f - p_{n-1}||_K ) 
  \leq
  \rho^n  {M ( 1 + R' ) \over {R'} ^n}.
$$
Since $\rho$ and $R'$ were arbitrary
numbers satisfying $1 < \rho < R' < R$, we
conclude that $ p_0 + \sum_1^{\infty} (p_n-p_{n-1})$ 
converges locally uniformly on $\Omega_R$ to a holomorphic function $F$. 
From (\ref{f=F}), $F = f$ on $K$.
\end{proof}

To verify (1) we will follow Bloom's reasoning in \cite{Bl}. Note that $Poly(nP)$ is a finite-dimensional complex vector space; we call its dimension $d_n$ (see Remark \ref{dn}). We begin with the key lemma. Fix $n\geq k$ where $k$ is as in (\ref{phyp}) and let $Q_1,...,Q_{d_n}$ be a basis for $Poly(nP)$. For $R>0$ define
$$D_R:=\{ z\in \CC^d: |Q_j(z)|< R^n, \ j=1,...,d_n\}.$$

\begin{lemma} \label{bloom2.3} Let $P$ be a lower set and let $f$ be holomorphic in a neighborhood of $\bar D_R$. Then for each positive integer $m$, there exists $G_m\in Poly(mP)$ such that for all $\rho \leq R$, 
$$||f-G_m||_{\bar D_{\rho}}\leq B(\rho/R)^m$$
where $B$ is a constant independent of $m$. 
\end{lemma}

\begin{proof} We show for $m=sn$, an integer multiple of $n$, that there exists $G_m\in Poly(mP)$ such that for all $\rho \leq R$, 
$$||f-G_m||_{\bar D_{\rho}}\leq B(\rho/R)^{m+n}$$
where $B$ is independent of $m$. To this end, let $S:\CC^d\to \CC^{d_n}$ via
$$S(z):=(Q_1(z),...,Q_{d_n}(z)).$$
Then $S(\CC^d)$ is a subvariety of $\CC^{d_n}$ and $S(D_R)$ is a subvariety of the polydisk
$$\Delta_R:=\{\zeta \in \CC^{d_n}: |\zeta_j|< R^n, \ j=1,...,d_n\}.$$
Choose $R_1>R$ so that $f$ is holomorphic on a neighborhood of $\bar D_{R_1}$. Let $F$ be holomorphic in a neighborhood of $\bar \Delta_{R_1}\subset \CC^{d_n}$ such that $F\circ S =f$ on $\bar D_{R_1}$. That such an $F$ exists follows from Theorem 8.2 in \cite{N}; see Remark \ref{nishino} below. Define
$$\beta_1:=R_1^n, \ \beta:=R^n, \ \hbox{and} \ \alpha:= \rho^n$$
Thus $\alpha \leq \beta <\beta_1$. Let
$$F(\zeta):=\sum_I F_I\zeta^I$$
be the Taylor series of $F$ about $0\in \CC^{d_n}$. By the Cauchy estimates on $\bar \Delta_{R_1}$, for each multiindex $I$ we have
$$|F_I|\leq ||F||_{\bar \Delta_{R_1}}\beta_1^{-|I|}.$$

Given a positive integer $s$, we let 
$$E_s(\zeta):= \sum_{|I|\leq s}F_I\zeta^I$$
be the Taylor polynomial of degree at most $s$ of $F$ at $0\in \CC^{d_n}$. Then for $m=sn$, let
$$G_m(z):=E_s\circ S(z)=\sum_{|I|\leq s} F_I Q_1(z)^{i_1}\cdots Q_{d_n}(z)^{i_{d_n}}$$
where $I=(i_1,...,i_{d_n})$. It follows that $G_m\in Poly(mP)$ because $Q_{j}\in Poly(nP)$ and $P$ is a lower set. Since $S(\bar D_{\rho})\subset \bar \Delta_{\rho}$, we have 
$$||f-G_m||_{\bar D_{\rho}}\leq ||F-E_s||_{\bar \Delta_{\rho}}\leq ||F||_{\bar \Delta_{R_1}}\sum_{|I|>s} (\frac{\alpha}{\beta_1})^{|I|}.$$
To obtain the desired estimate, note first that 
$$\sum_{|I|>s} (\frac{\alpha}{\beta_1})^{|I|}=\sum_{k=s+1}^{\infty} (\frac{\alpha}{\beta_1})^k{d_n+k-1\choose k}.$$
Choose $\beta_2$ with $\beta <\beta_2< \beta_1$ and $C>0$ so that
$$C(\beta_1/\beta_2)^k \geq {d_n+k-1\choose k}, \ k=1,2,3,...$$
Then
$$||f-G_m||_{\bar D_{\rho}}\leq ||F||_{\bar \Delta_{R_1}} \cdot C(\alpha /\beta_2)^{s+1} \cdot \frac{1}{1-\alpha/\beta_2}$$
$$ \leq B(\alpha/\beta)^{s+1}\leq B(\alpha/\beta)^{sn}$$
where $B=\frac{C||F||_{\bar \Delta_{R_1}}}{1-\beta/\beta_2}$.

\end{proof}

\begin{remark} \label{nishino} A version of this lemma was proved in \cite{Si} using the Oka extension theorem: instead of the mapping $S:\CC^d\to \CC^{d_n}$ via $S(z):=(Q_1(z),...,Q_{d_n}(z))$, one considers $\tilde S: \CC^d\to \CC^{d+d_n}$ via $\tilde S(z):=(z,Q_1(z),...,Q_{d_n}(z))$; the Oka result provides the existence of $\tilde F$ holomorphic on a neighborhood of $\tilde S(\bar D_R)$ with $\tilde F\circ \tilde S=f $ on a neighborhood of $\bar D_R$ (cf., section 3.3 of \cite{N}). We need to avoid using this Oka map $\tilde S$ as not all powers of the coordinates of $z\in \CC^d$ may be included in our $Poly(nP)$ spaces. Here, to apply the result in \cite{N}, we need $S$ to be {\it one-to-one} on $\bar D_R$. This follows since $n\geq k$ implies $\Sigma \subset nP$ so that the coordinate functions $e_j(z)=z_j, \ j=1,...,d$ belong to $Poly(nP)$ and $Q_1,...,Q_{d_n}$ form a basis for $Poly(nP)$.
\end{remark}

We want to construct polynomials $Q_1,...,Q_{d_n}$ so that for $n$ large the sets $D_R$ approximate the sublevel sets $\Omega_R$ of $V_{P,K}$. To this end, recall for $K\subset \CC^d$ compact, we defined
$$\Phi_n(z):= \sup \{|p_n(z)|: p_n\in Poly(nP),  \ \max_{\zeta \in K} ||p_n||_K\leq 1\}.$$
From Proposition \ref{turgay2}, we have for $K\subset \CC^d$ compact and nonpluripolar, 
$$V_{P,K} =\lim_{n\to \infty} \frac{1}{n} \log \Phi_n$$
pointwise on $\CC^d$; and if $\Phi:=e^{V_{P,K}}$ is continuous, the convergence is locally uniform on $\CC^d$. We assume continuity of $\Phi$ in Theorem \ref{BWthm}. We will use {\it Fekete points} and {\it Lagrange interpolating polynomials} to prove our results. To this end, let $\{e^{(n)}_j\}_{j=1,...,d_n}$ be basis monomials for $Poly(nP)$ where $d_n=$dim$(Poly(nP))$ and let $\{a_{nj}\}_{j=1,...,d_n}\subset K$ be {\it Fekete points of order $n$} for $K,Poly(nP)$; i.e., 
$$|VDM_n(a_{n1},...,a_{nd_n})|:=|\det [e^{(n)}_j(a_{nk})]_{j,k=1,...,d_n}|$$
is maximal among all $d_n-$tuples of points in $K$. Then the {\it fundamental Lagrange interpolating polynomials} 
$$l_j^{(n)}(z):=\frac{VDM_n(a_{n1},...,z,...,a_{nd_n})}{VDM_n(a_{n1},...,a_{nd_n})}, \ j=1,...d_n$$
($z$ in the $j-$th slot) form a basis for $Poly(nP)$ with the additional properties that
\begin{enumerate}
\item $||l_j^{(n)}||_K=1$; hence 
\item $\psi_n(z):=\max_{j=1,...,d_n}|l_j^{(n)}(z)|\leq \Phi_n(z)$ for $z\in \CC^d$; while
\item $\Phi_n(z)\leq d_n \psi_n(z)$ for $z\in \CC^d$.
\end{enumerate}

\noindent This final property follows from the Lagrange interpolation formula: for any $f$ defined on $K$, the Lagrange interpolating polynomial $L_n(f)$ for $f$ with nodes $a_{n1},...,a_{nd_n}$ is  
$$L_n(f)(z)=\sum_{j=1}^{d_n} f(a_{nj})l_j^{(n)}(z).$$
In particular, for $p_n\in Poly(nP)$, 
$$L_n(p_n)(z)=p_n(z)=\sum_{j=1}^{d_n} p_n(a_{nj})l_j^{(n)}(z).$$
Thus if, in addition, $||p_n||_K\leq 1$, 
$$|p_n(z)|\leq d_n \psi_n(z) \ \hbox{for}  \ z\in \CC^d.$$

For $R>0$ we have
$$\bar \Omega_R:=\{z\in \CC^d: \Phi(z) \leq R\}=\{z\in \CC^d: V_{P,K}(z)\leq \log R\}.$$
Defining 
\begin{equation}\label{dneqn} D_R:=\{ z\in \CC^d: \psi_n(z)<R^n\}=\{ z\in \CC^d: |l_j^{(n)}(z)|< R^n, \ j=1,...,d_n\},\end{equation} 
we have $\Omega_R\subset D_R$ since $\Phi=e^{V_{P,K}}\geq \Phi_n^{1/n}\geq \psi_n^{1/n}$. Note $D_R$ depends on $n$ while $\Omega_R$ does not. From (3), we get a reverse-type inclusion for $n$ large:

\begin{lemma} \label{bloom2.6} Given $0<R_1<R$, there exists $n_0$ such that for all $n\geq n_0$, 
$$D_{R_1} \subset \bar \Omega_R.$$

\end{lemma}

\begin{proof} We have 
$$\psi_n^{1/n} \geq d_n^{-1/n}\Phi_n^{1/n}.$$
Thus if $t<1$,
$$\psi_n(z)^{1/n}> t\Phi(z)$$
for $n\geq n_0$ where $n_0$ depends on $z$. Since we assume $\Phi$ is continuous, we can choose $n_0$ independent of $z$ for $z$ in a compact set; e.g., for $z\in \bar \Omega_{2R}$. Now take $t=R_1/R$.

\end{proof}

The following result proves the ``if'' direction of Theorem \ref{BWthm}.

\begin{proposition} Let $P$ be a lower set and let $K$ be compact and $PL-$regular. 
Let $R > 1$, and let $f$ be holomorphic on $\Omega_R$.
Then for any $R'<R$ the Lagrange interpolating polynomials $L_n(f)$ for $f$ associated with a Fekete array for $K,P$ satisfy
$$||f-L_n(f)||_K\leq B/(R')^n$$
where $B$ is a constant independent of $n$.
\end{proposition}

\begin{proof} Choose $R_1$ with $1<R_1<R'$. By Lemma \ref{bloom2.6} for all sufficiently large $n$ we have 
$$D_{R_1}\subset \bar \Omega_{R'}.$$
Here $D_{R_1}$ is defined in (\ref{dneqn}). Fix such an $n$. By Lemma \ref{bloom2.3} there exists $G_n\in Poly(nP)$ with
\begin{equation}\label{use} ||f-G_n||_{D_{\rho}}\leq B(\rho/R_1)^n \end{equation}
for all $\rho \leq R_1$. Let $g_n:= f-G_n$. Since $G_n\in Poly(nP)$, 
$$f-L_n(f)=g_n-L_n(g_n).$$
But
$$g_n(z)-L_n(g_n)(z)=g_n(z)-\sum_{j=1}^{d_n} g_n(a_{nj})l_j^{(n)}(z).$$
Recall that $|l_j^{(n)}(z)|\leq 1$ for all $z\in K$ and $j=1,...,d_n$. Thus, since $a_{nj}\in K\subset D_1$ for all $n$, from (\ref{use}), 
$$ |g_n(a_{nj})|\leq B(1/R_1)^n.$$
Thus for all $n$ sufficiently large, 
$$|\sum_{j=1}^{d_n} g_n(a_{nj})l_j^{(n)}(z)|\leq BR_1^{-n}d_n, \ z\in K.$$
Now since $K \subset \bar \Omega_1 \subset D_1$ (for all $n$), again by (\ref{use}) we have
$$|g_n(z)|\leq BR_1^{-n}, \ z\in K.$$
Thus for $n$ large we obtain
$$||f-L_n(f)||_K\leq BR_1^{-n} + BR_1^{-n}d_n$$
so that, indeed, 
$$||f-L_n(f)||_K\leq \tilde B/(R')^n.$$
\end{proof}

\section{Applications.}\label{sec:app} In this section we make the connection between Theorem \ref{BWthm} and Trefethen's work \cite{T}, where he introduces a new notion of degree for a polynomial 
$\displaystyle{p(x)=\sum_{\alpha\in \ZZ_+^d} a_\alpha x^\alpha}$, the {\it Euclidean} degree, which we may write as
\[{\rm deg}_E(p):=\max_{a_\alpha \neq 0}(|\alpha_1|^2+\cdots+|\alpha_d|^2)^{1/2}.\]

For $P\subset \RR_+^d$ a convex body, we may define an associated ``norm'' for $x\in \RR_+^d$ via the Minkowski functional
\[\|x\|_P:=\inf_{\lambda>0}\{x\in \lambda P\}.\] 
We remark that this defines a true norm on all of $\RR^d$ if $P$ is the positive ``octant'' of a centrally symmetric convex body $B,$ i.e., $P=B\cap (\RR_+)^d$. We may thus define a general degree associated to the convex body $P$ as
\[{\rm deg}_P(p):= \max_{a_\alpha \neq 0} \|\alpha\|_P.\]
Then
$$Poly(nP)=\{p\,:\,{\rm deg_P}(p)\le n\}.$$

For $q\ge1$, if we let 
\begin{equation}\label{Pq}
P_q:=\{(x_1,...,x_d): x_1,..., x_d\geq 0, \ x_1^q+\cdots +x_d^q\leq 1\}
\end{equation}
be the $(\RR_+)^d$ portion of an $\ell^q$ ball then we have, in the notation of \cite{T},
\begin{align*}
d_T(p)&={\rm deg}_{P_1}(p)\quad  \hbox{ ({\it total} degree)};\\
d_E(p)&={\rm deg}_{P_2}(p)\quad  \hbox{({\it Euclidean} degree)}; \\
d_{\rm max}(p)&={\rm deg}_{P_\infty}(p)\quad  \hbox{({\it max} degree)}.
\end{align*}

Further, if we let $1/q'+1/q=1$, then if $E_1,...,E_d\subset \CC$, from (\ref{product}),
\begin{align*}
V_{P_q,E_1\times \cdots \times E_d}(z_1,...,z_d)&=\|[V_{E_1}(z_1),V_{E_2}(z_2),
\cdots V_{E_d}(z_d)]\|_{\ell_{q'}}\\
&=
[V_{E_1}(z_1)^{q'}+\cdots +V_{E_d}(z_d)^{q'}]^{1/q'}.
\end{align*}

For the  particular product set, $K:=[-1,1]^d$ where $E_j=[-1,1]$ for $j=1,...,d,$ that Trefethen considers, $V_{E_j}(z_j)=\log|z_j+\sqrt{z_j^2-1}|$ and hence we have
\begin{equation}\label{VPq}
 V_{P_q,[-1,1]^d}(z_1,\cdots,z_d)=\left\{\sum_{j=1}^d \left(\log\left|z_j+\sqrt{z_j^2-1}\right|\right)^{q'}\right\}^{1/q'}.
 \end{equation}

Further, for $f$ continuous and complex-valued on $K$ we  define (as before) the approximation numbers,
\begin{align}\label{ApproxNums}
D_n(f,P,K)&:= \inf \{ ||f-p_n||_{K}:p_n \in Poly(nP)\} \nonumber \\
&=\inf \{ ||f-p_n||_{K}:{\rm deg}_{P}(p_n)\leq n\}.
\end{align} 

Essentially, \cite{T} compares approximation numbers $D_n(f,P_q,[-1,1]^d)$ for $q=1,$ $q=2$ and $q=\infty$ in different dimensions $d$ and notes the different rates of decay for holomorphic functions. Our Theorem \ref{BWthm} explains this behavior, precisely and in greater generality.

\begin{example}\label{example1}
Consider the multivariate Runge-type  function 
\[f(z):=\frac{1}{r^2+z^2},\quad r>0\]
where $z\in\CC^d,$ $z^2:=\sum_{j=1}^d z_j^2$
and $K=[-1,1]^d$ (cf. (2.1) of \cite{T}). This function is holomorphic except on its singular set 
\begin{align*}
S=S(f):=\{z\in\CC^d\,:\,z^2=-r^2\},
\end{align*}
an algebraic variety having no real points.

By Theorem \ref{BWthm}, the approximation numbers $D_n(f,P,K)$ decay like $R^{-n}$ iff $f$ is holomorphic in the set
\[\Omega_R:=\{z\in\CC^d\,:\, V_{P,K}(z)<\log R \}.\]
In other words,  $D_n(f,P,K)$ decays like $R^{-n}$ where 
\[R=R(P,K):=\sup\{R'>0\,:\, \Omega_{R'}\cap S=\emptyset \}.\]
It is easy to see that
\[ \log(R(P,K))= \min_{z\in S}V_{P,K}(z).\]

We note at this point the following elementary fact.
\begin{lemma}\label{Pscaled} Suppose that $c>0.$ Then
\[R(cP,K)=(R(P,K))^c.\]
\end{lemma}
\bigskip
We now compute the values of $R(P_q,K)$ for $q\ge1.$ Specifically
\begin{lemma}\label{q=1} For $q=1$ (corresponding to the total degree case) 
\[R(P_1,K)=\frac{r+\sqrt{r^2+d}}{\sqrt{d}}.\]
\end{lemma}
\begin{proof} In this case 
\[V_{P_1,K}(z)=\max_{1\le j\le d}\left\{\log\left|z_j+\sqrt{z_j^2-1}\right|\right\}.\]
Now, as is well known, the level sets of the univariate extremal function
$\log\left|\zeta+\sqrt{\zeta^2-1}\right|$ are confocal ellipses. Specifically, for $\rho>1,$
\[E_\rho:=\left\{\zeta\in\CC\,:\,\left|\zeta+\sqrt{\zeta^2-1}\right|=\rho\right\}\]
is the ellipse $(x/a)^2+(y/b)^2=1$ with $\zeta=x+iy$ and
$a=(\rho+1/\rho)/2,$ $b=(\rho-1/\rho)/2.$ The degenerate case with $\rho=1$
corresponds to $E_1=[-1,1].$ 

The interior and exterior  of the ellipse $E_\rho$ are given by the sublevel and suplevel sets
\begin{align*}
E_{<\rho}&:=\left\{\zeta\in\CC\,:\,\left|\zeta+\sqrt{\zeta^2-1}\right|<\rho
\right\}, \\
E_{>\rho}&:=\left\{\zeta\in\CC\,:\,\left|\zeta+\sqrt{\zeta^2-1}\right|>\rho
\right\}.
\end{align*}

For convenience, set $r'=r/\sqrt{d}$ so that the singular set
$$S(f)=\{z\in\CC^d\,:\,z^2=-d(r')^2\}.$$

For the particular case of $\rho=\rho^*:=
r'+\sqrt{(r')^2+1},$ it is easy to check that $a=\sqrt{1+(r')^2}$ and $b=r'.$ Hence
if $\zeta=x+iy\in E_{\le \rho^*},$  we have $|y|\le r'$ and $|y|=r'$ iff
$\zeta=\pm ir'.$ It follows that for $\zeta=x+iy\in E_{\le\rho^*},$ we have
\[\Re(\zeta^2)=x^2-y^2\ge -(r')^2\]
and $\Re(\zeta^2)= -(r')^2$ iff $\zeta=\pm ir'.$ Consequently, for a point $z$ on the singular set $S,$ i.e., with $\sum_{j=1}^d z_j^2=-d(r')^2$ (and hence 
$\sum_{j=1}^d\Re(z_j^2)=-d(r')^2$),  $z_j\in E_{<\rho^*}$ implies that
for some $k\neq j,$ $z_k\in  E_{>\rho^*}.$ Thus the minimum of $\exp(V_{P_1,K}(z))$ is 
$\rho^*=r'+\sqrt{1+(r')^2}=(r+\sqrt{r^2+d})/\sqrt{d},$ as claimed, and is attained for $z\in\{\pm ir/\sqrt{d}\}^d.$ 
\end{proof}
\bigskip

\begin{lemma} \label{CharCritPts} (Characterization of Lagrange Critical Points) Suppose that $q'=:p<\infty$ (i.e., $\infty\ge q>1$). Then, in the minimization problem
\[ \min_{z\in S} \exp\bigl(V_{P_q,K}(z)\bigr), \]
the critical points are characterized by the condition
\[m(z_j):=\bigl(\log(|z_j+\sqrt{z_j^2-1}|)\bigr)^{p-1}\frac{1}{z_j\sqrt{z_j^2-1}}
=m(z_{j'})\]
for every pair $1\le j,j'\le d.$
\end{lemma}
\begin{proof} We consider the objective function 
\[ F(x_1,y_1,\cdots,x_d,y_d):=\bigl(V_{P_q,K}(z)\bigr)^p=\sum_{j=1}^d f(z_j)^p\]
where $f(\zeta):=\log(|\zeta+\sqrt{\zeta^2-1}|),$ and separate the constraint
$\sum_{j=1}^dz_j^2=-r^2$ into its real and imaginary parts 
as
\begin{align*}
g_1(x_1,y_1,x_2,y_2,\cdots,x_d,y_d)&:=r^2+\sum_{j=1}^d (x_j^2-y_j^2)=0,\\
g_2(x_1,y_1,x_2,y_2,\cdots,x_d,y_d)&:=\sum_{j=1}^d x_jy_j = 0,
\end{align*}
where we have written $z_j=x_j+iy_j,$ $x_j,y_j\in\RR,$ $1\le j\le d.$

Then we may calculate
\[\nabla F=p\langle (f(z_1))^{p-1}\nabla f(z_1),\cdots,(f(z_d))^{p-1}\nabla f(z_d)\rangle \]
and
\begin{align*}
\nabla g_1&=2\langle x_1,-y_1,x_2,-y_2,\cdots,x_d,-y_d\rangle,\nonumber \\
\nabla g_2&=\langle y_1,x_1,y_2,x_2,\cdots,y_d,x_d\rangle.
\end{align*}
If we write the Lagrange multiplier conditions for a critical point as
\[ \frac{1}{p}\nabla F=\lambda_1 \frac{1}{2} \nabla g_1+\lambda_2\nabla g_2,\quad \lambda_1,\lambda_2\in\RR,\]
then critical points are characterized by
\[ (f(z_j))^{p-1}\nabla f(z_j)= \lambda_1 \langle x_j,-y_j\rangle
+\lambda_2  \langle y_j,x_j\rangle,\quad 1\le j\le d.\]
Treating the gradients as column vectors this latter condition may be expressed
in matrix form as
\[
(f(z_j))^{p-1}\nabla f(z_j)=
\begin{pmatrix} x_j & y_j \cr -y_j& x_j \end{pmatrix} 
\begin{pmatrix}\lambda_1 \cr \lambda_2\end{pmatrix},\quad 1\le j\le d
\]
iff 
\[ (f(z_j))^{p-1} \begin{pmatrix} x_j & y_j \cr -y_j& x_j \end{pmatrix} ^{-1} \nabla f(z_j)=\begin{pmatrix}\lambda_1 \cr \lambda_2\end{pmatrix},\quad 1\le j\le d
\]
iff 
\[(f(z_j))^{p-1}\frac{1}{x_j^2+y_j^2}
 \begin{pmatrix} x_j & -y_j \cr y_j& x_j \end{pmatrix} 
 \nabla f(z_j) =\begin{pmatrix}\lambda_1 \cr \lambda_2\end{pmatrix},\quad 1\le j\le d.\]
Consequently, 
\begin{equation}\label{CritPoint1}
(f(z_j))^{p-1} \frac{1}{|z_j|^2}
\begin{pmatrix} x_j & -y_j \cr y_j& x_j \end{pmatrix} 
\nabla f(z_j) = (f(z_{j'}))^{p-1} \frac{1}{|z_{j'}|^2}
\begin{pmatrix} x_{j'} & -y_{j'} \cr y_{j'}& x_{j'} \end{pmatrix} 
\nabla f(z_{j'})
\end{equation}
for $1\le j,j'\le d.$ 

We now proceed to calculate $\nabla f(\zeta)$. To this end, write
 \[f(\zeta)=\frac{1}{2}\log (|\zeta+\sqrt{\zeta^2-1}|^2) =\frac{1}{2}\log (h(\zeta)h(\overline{\zeta})) =\frac{1}{2}\{\log(h(\zeta))+\log(h(\overline{\zeta}))\}\]
with $h(\zeta):=\zeta+\sqrt{\zeta^2-1}.$ Note that
\[h'(\zeta)=1+\zeta/\sqrt{\zeta^2-1}=h(\zeta)/\sqrt{\zeta^2-1}.\]
Then
\begin{align*}
\frac{\partial f}{\partial x}&=\frac{1}{2}\left\{\frac{h'(\zeta)}{h(\zeta)}+\frac{h'(\overline{\zeta})}{h(\overline{\zeta})}\right\}\\
&=\frac{1}{2}\left\{\frac{1}{\sqrt{\zeta^2-1}}+\frac{1}{\sqrt{\overline{\zeta}^2-1}}\right\}\\
&=\Re\left\{\frac{1}{\sqrt{\zeta^2-1}}\right\}\\
\end{align*}
and, similarly,
\begin{align*}
\frac{\partial f}{\partial y}&= 
\frac{1}{2}\left\{i \frac{h'(\zeta)}{h(\zeta)}-i \frac{h'(\overline{\zeta})}{h(\overline{\zeta})}\right\}\\
&=-\Im \left\{\frac{1}{\sqrt{\zeta^2-1}}\right\}.
\end{align*}
Hence,
\[\nabla f(\zeta)= \begin{pmatrix}
\Re\left\{\frac{1}{\sqrt{\zeta^2-1}}\right\} \cr
-\Im \left\{\frac{1}{\sqrt{\zeta^2-1}}\right\}.
\end{pmatrix}
\]
Therefore
\begin{align*}
\frac{1}{|\zeta|^2} \begin{pmatrix} x & -y \cr y& x \end{pmatrix} 
 \nabla f(\zeta) &= 
\frac{1}{|\zeta|^2}\begin{pmatrix} x & -y \cr y& x \end{pmatrix} 
 \begin{pmatrix}
\Re\left\{\frac{1}{\sqrt{\zeta^2-1}}\right\} \cr
-\Im \left\{\frac{1}{\sqrt{\zeta^2-1}}\right\}
\end{pmatrix}\\
&=\frac{1}{|\zeta|^2}
\begin{pmatrix}  x\Re\left\{\frac{1}{\sqrt{\zeta^2-1}}\right\} 
+ y\Im \left\{\frac{1}{\sqrt{\zeta^2-1}}\right\} \cr
y \Re\left\{\frac{1}{\sqrt{\zeta^2-1}}\right\} -
x \Im \left\{\frac{1}{\sqrt{\zeta^2-1}}\right\} \end{pmatrix}\\
&=\frac{1}{|\zeta|^2} 
\begin{pmatrix}
\Re\left\{ \zeta \frac{1}{\sqrt{\overline{\zeta}^2-1}}\right\}\cr
\Im\left\{ \zeta \frac{1}{\sqrt{\overline{\zeta}^2-1}}\right\}\cr
\end{pmatrix}\\
&=\frac{1}{|\zeta|^2} \zeta \frac{1}{\sqrt{\overline{\zeta}^2-1}}\\
&=\frac{1}{\overline{\zeta}\sqrt{\overline{\zeta}^2-1}}.
\end{align*}

Substituting this into the critical point condition (\ref{CritPoint1}) and taking conjugates gives the result. \end{proof}

\bigskip
\begin{lemma} \label{q=inf} For $q=\infty,$ $p=q'=1,$ (corresponding to the tensor-product (max) degree case) 
\[R(P_\infty,K)=r+\sqrt{r^2+1}.\]
\end{lemma}
\begin{proof} 
By Lemma \ref{CharCritPts} for $p=1,$ critical points outside $[-1,1]^d$ are characterized by $z\in \CC^d$ such that
$z^2=-r^2$ and
\[\frac{1}{z_j\sqrt{z_j^2-1}}=
\frac{1}{z_{j'}\sqrt{z_{j'}^2-1}},\]
$1\le j,j'\le d.$  Squaring,  we see that it is necessary that
\[z_j^2(z_j^2-1)=z_{j'}^2(z_{j'}^2-1),\quad 1\le j,j'\le d.\]
However,
\[z_j^2(z_j^2-1)-z_{j'}^2(z_{j'}^2-1)=(z_j^2-z_{j'}^2)(z_j^2+z_{j'}^2-1)\]
and so either $z_j^2=z_{j'}^2$ or else $z_j^2+z_{j'}^2=1.$ 

To complete the proof  that the minimum of $F(z)=\exp(V_{P_\infty,K}(z))$ on the singular set is indeed 
$r+\sqrt{r^2+1},$ as claimed, we proceed by induction on the dimension $d.$ For dimension $d=1$ there is nothing to do. Hence, suppose that the result holds for  any $r>0$ and dimension strictly less that $d.$ We must show that it also holds in dimension $d.$ First note that for $z_1=ir$ and $z_2=z_3=\cdots=0,$ $F(z)=r+\sqrt{r^2+1}$ and hence $R(P_\infty,K)\ge r+\sqrt{r^2+1}.$
To show the reverse inequality there are three possibilities to consider:
\begin{enumerate}
\item $z\in S$ is a critical point for which $z_1^2=\cdots=z_d^2;$
\item $z\in S$ is a critical point for which there is a pair $j'\neq j$ such that
$z_j^2+z_{j'}^2=1;$
\item $z\in S$ is a boundary point, i.e, $z_j\in[-1,1]$ for some $j.$
\end{enumerate}

In case (1) we have $-r^2=\sum_{k=1}^dz_k^2=dz_j^2,$ i.e.,
\[z_j=\pm i r/\sqrt{d},\quad 1\le j\le d.\]
The value of $F(z)$ in this case is
\begin{equation}\label{CritValue}
F(z)=\prod_{j=1}^d |z_j+\sqrt{z_j^2-1}|=(r/\sqrt{d}+\sqrt{r^2/d+1})^{d}.
\end{equation}
But
\[r+\sqrt{r^2+1} \le (r/\sqrt{d}+\sqrt{r^2/d+1})^{d}\]
as the first term is the solution of the ODE $y'(r)=(1/\sqrt{r^2+1})y(r),$ $y(0)=1$ and the second of the ODE $y'(r)=(d/\sqrt{r^2+d})y(r),$ $y(0)=1$ with higher growth factor: $(d/\sqrt{r^2+d})>(1/\sqrt{r^2+1}).$ Thus such critical points are not candidates for the minimum.

In case (2) we may suppose that we have $z_{d-1}+z_d^2=1.$ Then for $z\in S,$
\[\sum_{j=1}^{d-2} z_j^2=-r^2-1=-(r')^2,\quad r':=\sqrt{r^2+1},\] and so by the $(d-2)$-dimensional case and the fact that $|z+\sqrt{z^2-1}|\ge1,$
\[F(z)\ge \min_{\sum_{j=1}^{d-2}z_j^2=-(r')^2}\prod_{j=1}^{d-2}
\left|z_j+\sqrt{z_j^2-1}\right|\ge r'+\sqrt{(r')^2+1}>r+\sqrt{r^2+1}.\]
Hence neither are such critical points candidates for the minimum.

Finally, for case (3), 
a boundary point has at least one of its coordinates in
$[-1,1].$ Without loss of generality we may assume that $z_d\in[-1,1].$ Then
$|z_d+\sqrt{z_d^2-1}|=1$ and
$z\in S$ iff $\sum_{j=1}^d z_j^2=-r^2,$ iff $\sum_{j=1}^{d-1}z_j^2=-r^2-z_d^2=-(r')^2$ with $r':=r^2+z_d^2\ge r^2.$ In other words
$z':=(z_1,\cdots,z_{d-1})\in\CC^{d-1}$ is on the $(d-1)$-dimensional singular set
$\sum_{j=1}^{d-1}z_j^2=-(r')^2.$ Hence, by the $(d-1)$-dimensional case the minimum of $F(z)$ (with $z_d\in[-1,1]$) is $r'+\sqrt{(r')^2+1}.$ Clearly, this is minimized for $z_d=0,$ in which case $r'=r$ and we are done. \end{proof}

\bigskip

\begin{lemma}\label{d=2Case}
Suppose that $d=2$ and let $p=q'.$ Then 
\[R(P_q,K)= r+\sqrt{r^2+1},\quad q\ge2\,\,(\hbox{i.e.}, \ p\le2).
\]
\end{lemma}
\begin{proof} We first prove the $p=2$ case.
By Lemma \ref{CharCritPts} the critical points are characterized by the condition
\begin{equation}\label{CritPtsp=2}
\log(|z_1+\sqrt{z_1^2-1}|)\frac{1}{z_1\sqrt{z_1^2-1}}
=\log(|z_{2}+\sqrt{z_{2}^2-1}|)\frac{1}{z_{2}\sqrt{z_{2}^2-1}}.
\end{equation} 
We may assume that
\[\log(|z_1+\sqrt{z_1^2-1}|)\ge \log(|z_{2}+\sqrt{z_{2}^2-1}|)\]
so that
\[\frac{z_1^2(z_1^2-1)}{z_2^2(z_2^2-1)}=\left(
\frac{\log(|z_1+\sqrt{z_1^2-1}|)}
{\log(|z_{2}+\sqrt{z_{2}^2-1}|)}\right)^2\ge 1.\]
But on the singular set $z_1^2=-r^2-z_2^2$ so we have
\[\frac{(r^2+z_2^2)(r^2+1+z_2^2)}{z_2^2(z_2^2-1)}\ge 1.\]
Setting $u=z_2^2,$ we have
\[\frac{(r^2+u)(r^2+1+u)}{u(u-1)}=\frac{u^2+(2r^2+1)u+r^2(r^2+1)}
{u^2-u}\ge 1,\]
which, upon dividing, becomes
\[1+\frac{2(r^2+1)u+r^2(r^2+1)}{u^2-u}\ge 1,\]
i.e., after dividing by $r^2+1,$
\[x:=\frac{2u+r^2}{u^2-u}\ge 0.\]
Then, cross-multiplying, we have
\[xu^2-(x+2)u-r^2=0,\quad x\ge0.\]
Since the discriminant of this quadratic is $D=(x+2)^2+4xr^2\ge0,$
it follows that $z_2^2=u\in\RR$ and hence $z_1^2=-r^2-z_2^2\in\RR,$ as well.

We now analyze the various possibilities for the minimum. Consider first
a boundary point where one of the coordinates, say $z_1,$ is in $[-1,1].$ In that case $\log(|z_1+\sqrt{z_1^2-1}|)=0,$ and 
\[\exp(V_{P_2,K}(z))=\exp(\|[0,\log(|z_2+\sqrt{z_2^2-1}|)]\|_{2})=
|z_2+\sqrt{z_2^2-1}|.\]
But \[z_2^2=-r^2-z_1^2=-(r')^2,\quad r':=\sqrt{r^2+z_1^2}\ge r\]
so that $z_2=\pm i r',$ and
\[|z_2+\sqrt{z_2^2-1}|=r'+\sqrt{(r')^2+1}\]
which is clearly minimized when $z_1=0$ in which case
\[\exp(V_{P_2,K}(z))=r+\sqrt{r^2+1}.\]
Thus $r+\sqrt{r^2+1}$ is the minimum value of $\exp(V_{P_2,K}(z))$ over boundary points.

Consider now the critical points. As reported above, in this case we must have
$z_1^2,z_2^2\in\RR.$ If say $z_1^2\ge0$ then just as in the boundary case
\[z_2^2=-r^2-z_1^2=-(r')^2,\quad r':=\sqrt{r^2+z_1^2}\ge r\]
and 
\begin{align*}
\exp(V_{P_2,K}(z))&\ge\exp(\|[0,\log(|z_2+\sqrt{z_2^2-1}|)]\|_{2})\\
&=|z_2+\sqrt{z_2^2-1}|\\
&\ge r+\sqrt{r^2+1}
\end{align*}
and so this case is not a candidate for the minimum. Hence we assume that
both $z_1^2<0$ and $z_2^2<0.$ Since the univariate extremal function is invariant under $z\mapsto -z,$ we may write $z_1=iy_1,$ $z_2=iy_2$ with $y_1,y_2\ge 0.$ The critical point condition (\ref{CritPtsp=2}) then reduces to
\[\log(y_1+\sqrt{y_1^2+1})\frac{1}{y_1\sqrt{y_1^2+1}}
= \log(y_2+\sqrt{y_2^2+1})\frac{1}{y_2\sqrt{y_2^2+1}}.\]
But the function
\[g(y):= \log(y+\sqrt{y^2+1})\frac{1}{y\sqrt{y^2+1}}=
\frac{1}{\sqrt{y^2+1}}\left(\frac{1}{y}\int_0^y \frac{1}{\sqrt{t^2+1}}\,dt\right)\]
is the product of two positive strictly decreasing functions; the function in parentheses being the average over the interval $[0,y]$ of a strictly decreasing function. Hence $g(y)$ is also strictly decreasing and the critical point condition therefore requires that $y_1=y_2.$ As $y_1^2+y_2^2=r^2,$ indeed
$y_1=y_2=r/\sqrt{2}.$

At this critical point 
\begin{align*}
\exp(V_{P_2,K}(z))&=\exp(\|[\log(r/\sqrt{2} + \sqrt{r^2/2+1}),\log(r/\sqrt{2} + \sqrt{r^2/2+1})]\|_{2})\\
&=\exp(2^{1/2}\log(r/\sqrt{2} + \sqrt{r^2/2+1}))\\
&=\bigl(r/\sqrt{2} + \sqrt{r^2/2+1}\bigr)^{\sqrt{2}}.
\end{align*}
We claim that this latter quantity is greater than $r+\sqrt{r^2+1}.$ Indeed,
taking logarithms, we claim that
\[\sqrt{2}\log(r/\sqrt{2} + \sqrt{r^2/2+1})\ge \log(r+\sqrt{r^2+1})\]
or, equivalently, that
\[\sqrt{2}\int_0^{r/\sqrt{2}}\frac{1}{\sqrt{t^2+1}}\,dt\ge 
\int_0^{r}\frac{1}{\sqrt{t^2+1}}\,dt.\]
Consider
\[G(y):=y\int_0^{r/y} \frac{1}{\sqrt{t^2+1}}\,dt\]
for $y\ge1.$
Then 
\[G'(y)= \int_0^{r/y} \frac{1}{\sqrt{t^2+1}}\,dt -\left(\frac{r}{y}\right)
\frac{1}{\sqrt{(r/y)^2+1}}\ge 0\]
as the integrand $1/\sqrt{t^2+1}$ is decreasing. Hence $G(y)$ is increasing and,
in particular, $G(\sqrt{2})\ge G(1),$ and the Lemma is proved for the $p=2$
case.

For $p\le 2$ ($q\ge2$), the monotonicity of $\ell_p$ norms implies that
\[V_{P_\infty,K}(z)\ge V_{P_q,K}(z)\ge V_{P_2,K}(z)\]
and so
\[\min_{z\in S} \exp(V_{P_\infty,K}(z)) \ge\min_{z\in S} \exp(V_{P_q,K}(z)) \ge \min_{z\in S} \exp(V_{P_2,K}(z)).\] 
From Lemma \ref{q=inf} and the $p=2$ case,
\[\min_{z\in S} \exp(V_{P_\infty,K}(z)) =\min_{z\in S} \exp(V_{P_2,K}(z))
=r+\sqrt{r^2+1}\]
and the result follows.
\end{proof}
As
\[R(P_1,K)=r/\sqrt{d}+\sqrt{1+r^2/d}<r+\sqrt{1+r^2}=R(P_2,K)=R(P_\infty,K)\]
the approximation order of the Euclidean degree is considerably higher than for the
total degree, while the use of tensor-product degree provides no additional advantage, as reported in \cite{T}.

It is also interesting to note that $R(P_1,K)$ decreases to $1$ as the dimension increases to $\infty$ while for $q\ge2$, $R(P_q,K)$  is independent of the dimension $d$ indicating that the rate of polynomial approximation using the total degree degenerates for higher dimensions  while for the Euclidean and tensor-product degree it does not. 

However this is not a completely fair comparison. The dimension of the spaces
$\{p\,:\,{\rm deg_P}(p)\le n\}$ are proportional (asymptotically) to the
volume ${\rm vol}_d(P);$ indeed, 
$$\hbox{dim}(\{p\,:\,{\rm deg_P}(p)\le n\})=\hbox{dim}(Poly(nP))\asymp {\rm vol}_d(P)\cdot n^d.$$ 
To equalize their dimensions we may scale $P_q$ by
\[c=c(q)=\left(\frac{{\rm vol}_d(P_1)}{{\rm vol}_d(P_q)}\right)^{1/d}.\]
For example, for $d=2,$ in the Euclidean case we have
\[c(2)=\left(\frac{1/2}{\pi/4}\right)^{1/2}=\sqrt{2/\pi}=0.7979\cdots.\]
By Lemma \ref{Pscaled} we then compare 
\[R(P_1,K)=r/\sqrt{2}+\sqrt{1+r^2/2}\,\,\hbox{and}\,\,
R(P_2,K)^{c(2)}=\bigl(r+\sqrt{1+r^2}\bigr)^{\sqrt{2/\pi}}.\]
We note that for ``small'' $r$ ($r<2.1090\cdots$) $R(P_2,K)^{c(2)}>R(P_1,K)$ and so the Euclidean degree, even in the dimension normalized case, has a better approximation order than the total degree case, albeit with a lesser advantage.
For example, for $r=0.25,$
\[R(P_1,K)=1.19228\cdots\,\,\hbox{and}\,\, R(P_2,K)^{\sqrt{2/\pi}}=1.2182\cdots.\]
Further, for $r$ ``large''  ($r>2.1090\cdots$), $R(P_1,K)>R(P_2,K)^{c(2)}$ so that then the {\it total} degree provides a better order of approximation.

\end{example}

\begin{example}\label{example2}
Consider now the bivariate function
\[f(z_1,z_2):=\frac{1}{(z_1-\alpha)^2+z_2^2}\]
for $\alpha\in\RR$ and $\alpha>1.$ This has a single {\it real} pole at $(z_1,z_2)=(\alpha,0)$ and complex singular set
\[S=S(f):=\{(z_1,z_2)\in\CC^2\,:\,z_2=\pm i(z_1-\alpha)\}.\]
\begin{lemma}\label{2ndExample} We have
\[R(P_q,K)=\begin{cases}
\alpha&\mbox{if } q=1\cr
\alpha-1+\sqrt{(\alpha-1)^2+1}&\mbox{if } q=\infty.
\end{cases} \]
\end{lemma}
\begin{proof} Consider first the $q=1$ case where
\[V_{P_1,K}(z)=\max_{1\le j\le 2}\log\bigl|z_j+\sqrt{z_j^2-1}\bigr|.\]
 We are claiming that
\[ \min_{z\in S} V_{P_1,K}(z)=\log(\alpha).\]
As discussed for Example 1, the level set  $\bigl|z_j+\sqrt{z_j^2-1}\bigr|=\alpha$ is  the ellipse $E_\alpha,$ $(x_j/a)^2+(y_j/b)^2=1$ with
\[a:=\frac{1}{2}\left(\alpha+\frac{1}{\alpha}\right),\,\,
b:=\frac{1}{2}\left(\alpha-\frac{1}{\alpha}\right). \]
Now if $z_1\in E_{\le\alpha}$ then 
\[\left(\frac{x_1}{a}\right)^2+\left(\frac{y_1}{b}\right)^2\le1\]
implies that $|x_1|\le\alpha$ and consequently that $|x_1-\alpha|\ge\alpha-a=b.$
But then  $z_2=\pm i(z_1-\alpha)=\pm (-y_1+i(x_1-\alpha))$ is such that
\[\left(\frac{y_1}{a}\right)^2+\left(\frac{x_1-\alpha}{b}\right)^2
\ge \left(\frac{x_1-\alpha}{b}\right)^2\ge1,\]
i.e., $z_2\in E_{>\alpha }.$

Similarly, one may show that if $z_2\in E_{<\alpha}$ then $z_1\in E_{>\alpha}.$
Consequently the minimum is for $z_1,z_2\in E_\alpha,$ for example
\[z_1=a=\frac{1}{2}\left(\alpha+\frac{1}{\alpha}\right),\,\,
z_2=-bi=-\frac{i}{2}\left(\alpha-\frac{1}{\alpha}\right).\]

We next consider the $q=\infty$ case where $q'=1$ and
\[V_{P_\infty,K}(z)=\sum_{j=1}^2 \log\bigl|z_j+\sqrt{z_j^2-1}\bigr|.\]
By calculations entirely analogous to those of the first example, the Lagrange multiplier critical points for $\min_{z\in S} \exp(V_{P_\infty,K}(z))$ are characterized by the condition that
\[\frac{1}{\sqrt{z_1^2-1}}=-i\frac{1}{\sqrt{z_2^2-1}}\]
from which it follows upon squaring that $z_1^2+z_2^2=2.$ Applying the constraint
$z_2=\pm i(z_1-\alpha)$ results in specific values for these critical points and their corresponding function values can be shown by elementary (but lengthy!) calculations to not be candidates for the minimum.

There remains the case of a boundary point, when of one of $z_1,z_2\in[-1,1].$
If $z_1\in[-1,1]$ then $\log\bigl|z_1+\sqrt{z_1^2-1}\bigr|=0$ and then for $z_2=-i(z_1-\alpha),$
$\log\bigl|\alpha-z_1+\sqrt{(\alpha-z_1)^2+1}\bigr|$
is minimized by $z_1=1$ for which
\[\exp(V_{P_\infty,K}(z))=\alpha-1+\sqrt{(\alpha-1)^2+1}.\]
On the other hand, if $z_2\in[-1,1],$ then $\log\bigl|z_2+\sqrt{z_2^2-1}\bigr|=0$
and $z_1=\alpha+iz_2.$ It is easy to see that then 
$\log\bigl|z_1+\sqrt{z_1^2-1}\bigr|$ is minimized for $z_2=0,$ i.e., $z_1=\alpha.$
But the ellipse $E_\rho$ with $\rho=\alpha-1+\sqrt{(\alpha-1)^2+1}$ has semi-major axis $a=\sqrt{(\alpha-1)^2+1}<\alpha,$ for $\alpha>1.$ Hence $z_1\in E_{>\rho}$ and, in this case,
\[\bigl|z_1+\sqrt{z_1^2-1}\bigr|>\alpha-1+\sqrt{(\alpha-1)^2+1}\]
and this is not a candidate for the minimum.

\end{proof}
Again we have $R(P_1,K)<R(P_\infty,K)$ (note that, by the monotonicity of $\ell_p$ norms, we have $R(P_2,K)\le R(P_\infty)$ and so, at best, $R(P_2,K)= R(P_\infty,K)$). However the gain in approximation order is much less. Indeed, if we write $\alpha=1+\epsilon,$ $\epsilon>0,$ then 
\[R(P_1,K)=1+\epsilon\,\,\hbox{and}\,\,R(P_\infty,K)=\epsilon+\sqrt{1+\epsilon^2}=1+\epsilon +\epsilon^2/2+\cdots.\]
Further, if we normalize the area of $P$ to make the dimesnions of the spaces comparable, we obtain
\[R(P_1,K)=1+\epsilon > (\epsilon +\sqrt{1+\epsilon^2})^{\sqrt{2/\pi}}=1+\sqrt{2/\pi}\epsilon +\cdots\ge
R(P_2,K) \]
even for small $\epsilon.$ In other words, the total degree is then, in this sense, the better option.
\end{example}
\begin{example}\label{example3} In Example \ref{example1}, and from the numerical evidence, also in Example \ref{example2}, it is the case that
\[R(P_2,K)=R(P_\infty,K)\]
indicating that there is no advantage in using the tensor-product degree over the Euclidean degree. This is {\it not} always the case. Indeed, consider the function
\[f(z_1,z_2)=\frac{1}{(z_1-\alpha)^2+(z_2-\alpha)^2},\quad \alpha>1.\]
We report the numerical result that
for $\alpha=5/4$ we obtain
\[R(P_2,K)=2.0518< 2.1531 =R(P_\infty,K).\]

\end{example}


\begin{thebibliography}{GGK2}

\bibitem{BagLe} T. Bagby and N. Levenberg, Bernstein theorems, \emph{New Zealand J. Math.}, \textbf{22} (1993), no. 1, 1-20.

\bibitem {Bay} T. Bayraktar, Zero distribution of random sparse polynomials, arXiv:1503.00630v4, to appear in \emph{Mich. Math. J.}

\bibitem{Bl} T. Bloom, On the convergence of multivariable Lagrange interpolants, \emph{Constr. Approx.}, \textbf{5} (1989), 415-435.

\bibitem{PELD} T. Bloom and N. Levenberg, Pluripotential energy and large deviation, \emph{ Indiana Univ. Math. J.}, \textbf{62} (2013), no. 2, 523-550.

\bibitem{K} M. Klimek, \emph{Pluripotential Theory}, Oxford Univ. Press, 1991.

\bibitem{N} T. Nishino, \emph{Function theory in several complex variables}, AMS, Providence, RI, 1996.

\bibitem{ST} E. Saff and V. Totik, \emph{Logarithmic potentials with external fields}, Springer-Verlag, Berlin, 1997.

\bibitem{Si} J. Siciak, Extremal plurisubharmonic functions in $\CC^N$, {\emph Ann. Polon. Math.}
{\textbf 39} (1981), 175-211.

\bibitem{T} N. Trefethen, Multivariate polynomial approximation in the hypercube, submitted to \emph{PAMS}.

\end{thebibliography}
\end{document}